\documentclass[12pt]{amsart}
\usepackage{amsmath,amssymb,  latexsym,amsthm,epsfig,xspace,url}
\usepackage{amsxtra}
\usepackage{amsfonts} 
\usepackage{color}
\usepackage[all]{xy}
\usepackage{enumitem}
\usepackage{tikz} 

\newcommand{\mc}[1]{{}}

\newcommand {\bi} {{\bf i}} 

\DeclareMathRadical{\sqrtsign}{symbols}{"70}{largesymbols}{"70}
\newcommand{\bb}{\mathbb}

\newcommand{\cx}{{\bb C}}

\newcommand{\integers}{{\bb Z}}
\newcommand{\natls}{{\bb N}}

\newcommand{\reals}{{\bb R}}

\newlength{\figboxwidth}             
\renewcommand{\bold}[1]{\medskip \noindent {\bf #1 }\nopagebreak}

\newcommand{\dirsum}{\oplus}

\newcommand{\tensor}{\otimes}

\newcommand{\cross}{\times}

\newcommand{\st}{\;\: : \;\:}         

\newcommand{\zed}{\integers}

\newcommand{\Sp}{\operatorname{Sp}}

\def\@ifundefined#1#2#3%
  {\expandafter\ifx\csname#1\endcsname\relax#2\else#3\fi}

\@ifundefined{theoremstyle}{
}{
\theoremstyle{plain} 
}
\newtheorem{theorem}{Theorem}[section]

\newtheorem{proposition}[theorem]{Proposition}
\newtheorem{lemma}[theorem]{Lemma}

\newtheorem{corollary}[theorem]{Corollary}

\@ifundefined{theoremstyle}{
}{
\theoremstyle{definition} 
}
\newtheorem{definition}[theorem]{Definition}

\newcommand{\cG}{{\mathcal G}}
\newcommand{\cH}{{\mathcal H}}

\newcommand{\cL}{{\mathcal L}}
\newcommand{\cM}{{\mathcal M}}
\newcommand{\cN}{{\mathcal N}}

\mathchardef\GG="321D
\catcode`~=11 
\newcommand{\urltilde}{\kern -.15em\lower .7ex\hbox{~}\kern .04em}
\catcode`~=13 

\numberwithin{equation}{section}
\numberwithin{theo}{section}
\newcommand{\noz}{n}

\title[Symplectic and Isometric Subbundles]{Symplectic and Isometric
  $SL(2,\reals)$-invariant subbundles of the Hodge bundle}

\author{Artur Avila}

\address{
CNRS, IMJ-PRG, UMR 7586, Univ Paris Diderot, Sorbonne Paris Cit\'e,
Sorbonnes Universit\'es, UPMC Univ Paris 06, F-75013, Paris, France
}
\address{IMPA,
Estrada Dona Castorina 110, Rio de Janeiro, Brasil} 
\email{artur@math.univ-paris-diderot.fr}

\thanks{Research of the first author was supported by the ERC Starting Grant ``Quasiperiodic'' and by the Balzan project of Jacob Palis.}

\author{Alex Eskin}

\address{Alex Eskin: Department of Mathematics, University of Chicago,
  Chicago, IL 60637, USA}
\email{eskin@math.uchicago.edu}
\urladdr{http://www.math.uchicago.edu/~eskin}

\thanks{Research  of  the second author is partially supported  by
NSF grants DMS 0604251, DMS 0905912 and DMS 1201422
}

\author{Martin M\"oller}

\address{Martin M\"oller: Institut f\"ur Mathematik, Goethe-Universit\"at Frankfurt, Robert-Mayer-Str. 6-8, 
60325 Frankfurt am Main, Frankfurt, Germany.}
\email{moeller@math.uni-frankfurt.de.}
\urladdr{https://www.uni-frankfurt.de/fb/fb12/mathematik/ag/personen/moeller}

\thanks{Research of the third author is partially
supported by ERC grant 257137 ``Flat surfaces''}

\begin{document}

\maketitle

\section{Introduction}

Suppose $g \ge 1$, and 
let $\kappa = (\kappa_1,\dots, \kappa_\noz)$ be a partition of $2g-2$,
and 
let $\cH(\kappa)$ be a stratum of Abelian differentials,
i.e. the space of pairs $(M,\omega)$ where $M$ is a Riemann surface
and $\omega$ is a holomorphic $1$-form on $M$ whose zeroes have
multiplicities $\kappa_1 \dots \kappa_\noz$. The form $\omega$ defines a
canonical flat metric on $M$ with conical singularities at the zeros
of $\omega$. Thus we refer to points of $\cH(\kappa)$ as
{\em flat surfaces} or {\em translation surfaces}. For an introduction
to this subject, see the survey \cite{Zorich:survey}. 

The space $\cH(\kappa)$
admits an action of the group $SL(2,\reals)$ which generalizes the
action of $SL(2,\reals)$ on the space $GL(2,\reals)/SL(2,\zed)$ of flat
tori.

\bold{Affine measures and manifolds.}
The area of a translation surface is given by 
\begin{displaymath}
a(M,\omega) = \frac{i}{2} \int_M \omega \wedge \bar{\omega}.
\end{displaymath}
A ``unit hyperboloid'' $\cH_1(\kappa)$
is defined as a subset of translation surfaces in $\cH(\kappa)$ of
area one. 
For a subset $\cN_1 \subset \cH_1(\kappa)$ we write
\begin{displaymath}
\reals \cN_1 = \{ (M, t \omega) \;|\; (M,\omega) \in \cN_1, \quad t \in
\reals \} \subset \cH(\kappa).
\end{displaymath}
\begin{definition}
\label{def:affine:measure}
An ergodic $SL(2,\reals)$-invariant probability measure $\nu_1$ on
$\cH_1(\kappa)$ is called {\em affine} if the following hold:
\begin{itemize}
\item[{\rm (i)}] The support $\cN_1$ of $\nu_1$ is a suborbifold of
  $\cH_1(\kappa)$. Locally in period coordinates (see
  \S\ref{sec:forni:subspace} below), the suborbifold
  $\cN = \reals \cN_1$
  is defined as subset of $\cx^n$ 
by complex linear equations with real coefficients.  
\item[{\rm (ii)}] Let $\nu$ be the measure supported on $\cN$ so that $d\nu =
    d\nu_1 da$. Then $\nu$ is an affine linear measure in the period
    coordinates on $\cN$, i.e. it is (up to normalization) the
    restriction of Lebesgue measure to the subspace $\cN$. 
\end{itemize}
\end{definition}

\begin{definition}
\label{def:affine:invariant:submanifold}
We say that a suborbifold $\cN_1$ for which there exists a measure
$\nu_1$ such that the pair $(\cN_1, \nu_1)$ 
satisfies (i) and (ii) is an {\em affine invariant submanifold}. 
\end{definition}
Note that in particular, any affine invariant submanifold is a closed
subset of $\cH_1(\kappa)$ which is invariant under the $SL(2,\reals)$
action, and which in period coordinates looks like an affine subspace. 

We also consider the entire stratum $\cH(\kappa)$ to be an
(improper) affine invariant submanifold.

\bold{The tangent space of an affine submanifold.}
Suppose $\cN$ is an affine invariant submanifold. Then, by definition,
in period coordinates the tangent bundle $T_\cN$ of $\cN$ is determined
by a subspace $T_\cx(\cN)$ of the vector space the ambient manifold
is modeled on. Condition i) implies moreover that this
subspace is of the form 
\begin{displaymath} 
T_\cx(\cN) = \cx \tensor T_\reals(\cN),
\end{displaymath}
where $T_\reals(\cN) \subset H^1(M,\Sigma,\reals)$.
Let $p: H^1(M, \Sigma, \reals) \to H^1(M,\reals)$ be the natural map. 
We can then consider the 
subspace $p(T_\reals(\cN)) \subset H^1(M,\reals)$. 

\bold{The Forni subspace.} Let $\nu$ be a finite $SL(2,\reals)$-invariant
measure on $\cH_1(\kappa)$. For $\nu$-almost all $x \in \cH(\kappa)$, 
let $F(x) \subset H^1(M,\reals)$ be the
maximal $SL(2,\reals)$-invariant subspace on which the
Kontsevich-Zorich cocycle acts by isometries in the Hodge inner
product (see \S\ref{sec:forni:subspace} below). Then, the subspaces
$F(x)$ form an $SL(2,\reals)$-invariant subbundle of the Hodge bundle.

For the Masur-Veech (i.e. Lebesgue) measure on $\cH_1(\kappa)$, the Forni
subspace $F(x)
= \{ 0 \}$ almost everywhere. However, there exist  
affine $SL(2,\reals)$-invariant measures $\nu$ for which  $F(x) \ne 0$
for $\nu$-almost-all $x$, see e.g. \cite{Forni:handbook}, 
\cite{Forni:Matheus}. 

\bold{Terminology.} By the term Hodge bundle over an affine manifold $\cN$
we mean the (flat) vector bundle
$H^1_\reals$ with fiber $H^1(M,\reals)$ over the point $(M,\omega) \in
\cN$. By a flat subbundle, 
we mean a subbundle which is flat with respect to the Gauss-Manin
connection.  By an $SL(2,\reals)$-invariant subbundle we mean a
subbundle which is fiberwise equivariant with respect to the
$SL(2,\reals)$ action. Note that flat subbundles are automatically
$SL(2,\reals)$-invariant, but there exist $SL(2,\reals)$-invariant
subbunles which are not flat. 

\bold{The main results.} In this note, we prove the following:
\begin{theorem}
\label{theorem:forni:transverse:to:tangent}
Let $\nu$ be an affine measure on $\cH_1(\kappa)$ 
and let $\cN$ be affine submanifold on
which $\nu$ is supported. Then, 
\begin{itemize}
\item[{\rm (a)}] Except for a set of $\nu$-measure $0$, $F(x)$ is
  locally constant on $\cN$, and thus defines a flat subbundle of the
  Hodge bundle $H^1_\reals$ over $\cN$.  
\item[{\rm (b)}] For $\nu$-almost-all $x$,
$p(T_\reals(\cN))(x)$ is orthogonal to $F(x)$ with respect to the
Hodge inner product at $x$. 
\item[{\rm (c)}] For $\nu$-almost-all $x$,
$p(T_\reals(\cN))(x)$ is orthogonal to $F(x)$ with respect to the
intersection form. 
\end{itemize}
\end{theorem}

As corollaries, we get the following:
\begin{theorem}
\label{theorem:affine:symplectic}
Any affine manifold $\cN$ is symplectic, in the sense that the intersection
form is non-degenerate on $p(T_\reals(\cN))$. 
\end{theorem}

\begin{theorem}
\label{theorem:locally:constant:semisimple}
The Hodge bundle $H^1_\reals$ 
over any affine manifold $\cN$ is semisimple, in the
sense that any flat subbundle has
a complementary flat subbundle. 
\end{theorem}

However the main motivation for this paper is that a somewhat more 
technical version of
Theorem~\ref{theorem:forni:transverse:to:tangent} where one does not
assume that $\nu$ is an affine measure, see
Theorem~\ref{theorem:technical} below, is needed in
\cite{Eskin:Mirzakhani:measures} to complete the proof of the fact
that all $SL(2,\reals)$-invariant measures are
affine. (Theorem~\ref{theorem:technical} is only needed in
\cite{Eskin:Mirzakhani:measures} in the presence of relative homology). 

\bold{Organization of the paper.} 
In \S\ref{sec:forni:subspace}, we
recall the definitions of the Kontsevich-Zorich cocycle and the Forni
subspace, and make some preliminary statements. In
\S\ref{sec:real:analytic:envelope}, we define the ``real-analytic
envelope'', which is local real-analytic
version of the Zariski closure of the support of an
$SL(2,\reals)$-invariant measure, and prove that it is locally
affine. (This construction is only needed for the proof of
Theorem~\ref{theorem:technical}; in that case, in view of the fact
that the Forni subspace is a real-analytic object, we use it to effectively
replace the support of the measure by its real-analytic envelope, 
which is an affine subspace). 

In \S\ref{sec:forni:revisited}, we state some additional local
properties of the Forni subspace. In \S\ref{sec:connection}, we define
a connection along unstable (and stable) submanifolds which is
different from the Gauss-Manin connection; we then show that the Forni
subspace is equivariant and isometric 
with respect to this connection, and that the
restriction of the connection to the Forni subspace is
real-analytic. This allows us, in \S\ref{sec:preliminary:formula}, to
prove a preliminary formula for the variation of the Forni subspace
along stable and unstable leaves. Finally, in \S\ref{sec:curvature},
we use this formula to 
compute the parallel transport a vector in the Forni subspace along a closed
path composed of segments along stable and unstable leaves, and show that
the resulting monodromy map is unipotent. Since the monodromy must
also take values in a compact group, this shows that the monodromy map
is the identity, from which we deduce that the Forni subspace is
flat. 

We note that, from the point of view of partially hyperbolic dynamics,
the connection we use is given by the holonomy along the strong
foliations of a fiber bunched cocycle.  The calculation in this paper
arises naturally when showing that the projective action of the
groupoid generated by the holonomies does not preserve some measure, a
property which is commonly used to establish the non-triviality of
Lyapunov spectra, see e.g. \cite{Bonatti:Gomez-Mont:Viana},
\cite{Avila:Viana} and references therein.

\bold{Acknowledgements.} The authors would like to thank Alex Wright
for making several useful suggestions, in particular regarding the
formulation of Theorem~\ref{theorem:forni:transverse:to:tangent} and
Theorem~\ref{theorem:locally:constant:semisimple}. 
We would also like
to thank the referee for his careful reading of the paper and for
making numerous detailed comments 
which have helped us to greatly improve the
presentation. 

\section{The Kontsevich-Zorich Cocycle}
\label{sec:forni:subspace}

\bold{Algebraic Hulls.} The algebraic hull of a cocycle is defined in
\cite{ZimmerBook}. We quickly recall the definition: 
Suppose a group $G$ acts on a space $X$, preserving a measure $\nu$, 
and suppose $H$ is an $\reals$-algebraic group. 
Let $A: G \cross X \to H$ be a cocycle, i.e. $A$ is a measurable map,
such that $A(g_1 g_2 x) = A(g_1, g_2 x) A(g_2, x)$.  We say that the
$\reals$-algebraic subgroup $H'$ of $H$ is the {\em algebraic hull} of
$A$ if $H'$ is the smallest $\reals$-algebraic subgroup of $H$ such
that there exists a measurable map $C: X \to H$ with the property that
\begin{displaymath}
C( g x)^{-1} A(g,x) C(x) \in H' \quad \text{for almost all $g \in G$
  and almost all $x \in X$. }
\end{displaymath}
It is shown in \cite{ZimmerBook} \mc{give exact reference} that the
algebraic hull exists and is unique up to conjugation.

\bold{Period Coordinates.}
Let $\Sigma \subset M$ denote the set of zeroes of $\omega$. Let $n =
|\Sigma|$, and let $k = 2g + n-1$. Let $\{\gamma_1, \dots, \gamma_{2g}
\}$ be a symplectic $\zed$-basis for the homology group $H_1(M,\zed)$,
and let
$\{ \gamma_1, \dots, \gamma_k\}$ denote its extension to  a 
$\zed$-basis for the relative 
homology group $H_1(M,\Sigma, \zed)$. We can define a map $\Phi:
\cH(\kappa) \to \cx^k$ by 
\begin{displaymath}
\Phi(M,\omega) = \left( \int_{\gamma_1} \omega, \dots, \int_{\gamma_k}
  w \right)
\end{displaymath}
The map $\Phi$ (which depends on a choice of the basis $\{ \gamma_1,
\dots, \gamma_k\}$) is a local coordinate system on the stratum
$\cH(\kappa)$ near $(M,\omega)$.  
Alternatively,
we may think of the cohomology class $[\omega] \in H^1(M,\Sigma, \cx)$ as a
local coordinate on the stratum $\cH(\kappa)$. We
will call these coordinates {\em period coordinates}. 

\bold{The $SL(2,\reals)$-action and the Kontsevich-Zorich cocycle.}
We write $\Phi(M,\omega)$ as a $2$ by $n$ matrix $x$. The action of $g =
\left(\begin{smallmatrix} a & b \\ c & d \end{smallmatrix} \right) \in
SL(2,\reals)$ in these coordinates is linear. We choose some
fundamental domain for the action of the mapping class group, and
think of the dynamics on the fundamental domain. Then, the
$SL(2,\reals)$ action becomes
\begin{displaymath}
x = \begin{pmatrix} x_1 & \dots & x_k \\ y_1 & \dots & y_k \end{pmatrix}
\to gx = \begin{pmatrix} a & b \\ c & d \end{pmatrix} \begin{pmatrix} x_1 & \dots & x_k \\ y_1 & \dots & y_k
\end{pmatrix} A(g,x),
\end{displaymath}
where $A(g,x) \in \Sp(2g,\zed) \ltimes \reals^{n-1}$ is the {\em Kontsevich-Zorich
cocycle}. Thus, $A(g,x)$ is change of basis one needs to perform to return the
point $gx$ to the fundamental domain. It can be interpreted as the
monodromy of the Gauss-Manin connection (restricted to the orbit of
$SL(2,\reals)$).

The following theorem is essentially due to Forni \cite{Forni:Deviation}, and
Forni-Matheus-Zorich \cite{Forni:Matheus:Zorich:Lyapunov}. It is stated as
\cite[Theorem~A.6]{Eskin:Mirzakhani:measures}. 
For a self-contained proof (which
essentially consists of references to \cite{Forni:Matheus:Zorich:Lyapunov}) 
see Appendix~A of
\cite{Eskin:Mirzakhani:measures}. 

\begin{theorem}
\label{theorem:KZ:semisimple}
Let $\nu$ be an ergodic $SL(2,\reals)$-invariant probability measure. Then, 
\begin{itemize}
\item[{\rm (a)}] The $\nu$-algebraic hull $\cG$ of the
  Kontsevich-Zorich cocycle is semi\-simple.
\item[{\rm (b)}] On any connected finite cover of $\cH_1(\kappa)$,
  each $\nu$-measurable \mc{strongly???} irreducible $SL(2,\reals)$-invariant
   subbundle of the Hodge bundle is either
  symplectic or isotropic. 
\end{itemize}
\end{theorem}

\bold{The Forni subspace.} 

\begin{definition}[Forni Subspace]
\label{theorem:Forni:analytic:def}
Let
\begin{equation}
\label{eq:Forni:analytic:def}
F(x) = \bigcap_{g \in SL(2,\reals)} g^{-1} (\operatorname{Ann}
B^{\reals}_{gx}), 
\end{equation}
where for $(M,\omega) \in \cH_1(\kappa)$ the quadratic form
$B_\omega^\reals( \cdot, \cdot )$ is as defined in
\cite{Forni:Matheus:Zorich:Lyapunov} by
\begin{displaymath}
B_\omega^\reals( \alpha, \beta ) = \int_M \alpha \beta \frac{\overline{\omega}}{\omega}.
\end{displaymath}
\end{definition}

\bold{Remark.} The form $B_\omega^\reals( \cdot, \cdot )$ measures the
derivative of the corresponding period matrix entry along a
Teichm\"uller deformation in the direction of the quadratic differential
$\omega^2$ (Ahlfors-Rauch variational formula).
\par
It is clear from the definition, that as long as its
dimension remains constant, $F(x)$ varies real-analytically with $x$.  

\begin{theorem}
\label{theorem:forni:subspace:ergodic:def}
Suppose $\nu$ is an ergodic
$SL(2,\reals)$-invariant probability 
measure. Then the subspaces $F(x)$ where $x$
varies over the support of $\nu$ form the maximal 
$\nu$-measurable $SL(2,\reals)$-invariant 
isometric subbundle of the Hodge
bundle. 
\end{theorem}


\bold{Proof.} Let $F(x)$ be as defined in
(\ref{eq:Forni:analytic:def}). Then, $F$ is an
$SL(2,\reals)$-invariant subbundle of the Hodge bundle, 
and the restriction
of $B_x^\reals$ to $F(x)$ is identically $0$. Consequently, by \cite[Lemma
1.9]{Forni:Matheus:Zorich:Lyapunov}, $F$ is isometric. 

Now suppose $N$ is any other $\nu$-measurable 
isometric $SL(2,\reals)$-invariant subbundle of
the Hodge bundle. Then by \cite[Theorem
2]{Forni:Matheus:Zorich:Lyapunov}, $N(x) \subset \operatorname{Ann}
B_x^\reals$. Since $N$ is $SL(2,\reals)$-invariant, we have $N \subset
F$. Thus $F$ is maximal. 
\qed\medskip

\begin{theorem}
\label{theorem:properties:forni}
On any finite connected cover of $\cH_1(\kappa)$, the following statements hold.
\begin{itemize}
\item[{\rm (a)}] The Forni subspace is symplectic, and its symplectic
  complement $F^\dagger$ coincides with its Hodge complement $F^\perp$. 

\item[{\rm (b)}] Any $SL(2,\reals)$-invariant subbundle of $F^\perp$
  is symplectic, 
  and the restriction of the Kontsevich-Zorich cocycle to
  any $SL(2,\reals)$-invariant subbundle of $F^\perp$ has at least one non-zero
  Lyapunov exponent. 
\end{itemize}
\end{theorem}

\bold{Proof.} See \cite[Theorem~A.9]{Eskin:Mirzakhani:measures} (the
proof of which consists of references to \cite{Forni:Matheus:Zorich:Lyapunov}). 
\qed\medskip

\bold{Remark.} In view of Theorem~\ref{theorem:properties:forni}, the
Forni subspace corresponds to the maximal compact factor of the
algebraic hull of Kontsevich-Zorich 
cocycle over the action of $SL(2,\reals)$. By definition, all
Lyapunov exponents on the Forni subspace are zero. However, the
(non-$SL(2,\reals)$-invariant) zero-Lyapunov subspace of the
Kontsevich-Zorich cocycle  may be larger, see
\cite{Forni:Matheus:Zorich:Zero} for an example.

\section{The Real Analytic Envelope}
\label{sec:real:analytic:envelope}

In this section, we define a local real-analytic version of the
Zariski closure of the support of an $SL(2,\reals)$-invariant
probability measure. To study the Forni subspace, we must work in the
real-analytic category, since the Forni subspace is a real-analytic
object.

Let $\nu$ be an ergodic $SL(2,\reals)$-invariant probability measure on the
stratum. We break up the stratum into charts, and consider each chart
as a subset of $\cx^n$. 

Let $B(x,\epsilon)$ denote the ball centered at $x$ of radius
$\epsilon$. \mc{in what metric}
Let $N(x,\epsilon)$ be the smallest real-analytic subset (in the sense of
\cite[Definition I.1]{Narasimhan}) of $B(x,\epsilon)$ 
such that $\nu(B(x,\epsilon) \cap N(x,\epsilon)^c) = 0$, where
$N(x,\epsilon)^c$ denotes the complement of $N(x,\epsilon)$. 
Such an $N(x,\epsilon)$ exists by
\cite[Corollary V.2]{Narasimhan}.  Note that if $x$ is disjoint from
the support of $\nu$, $N(x,\epsilon)$ will be the
empty set for any sufficiently small $\epsilon$. 

Let $N$ be a real analytic set. 
A point $y \in N$ is called regular if, near $y$, the set $N$ is a
real-analytic submanifold of $\cx^n$. Let $N_{reg}$ denote the set of
regular points of $N$ and let $N_{sing}$ denote the set of singular
points. 
\par
\begin{lemma}
\label{lemma:N:stabilizes}
There exists $\epsilon_0 = \epsilon_0(x)$, such that the following conditions hold:
\begin{itemize}
\item[i)] For all $\epsilon < \epsilon_0$
\begin{displaymath}
N(x,\epsilon) = N(x,\epsilon_0) \cap B(x,\epsilon). 
\end{displaymath}
\item[ii)] If  $S \subset B(x,\epsilon_0)$ is a  real-analytic subset for which the inclusion 
$S_x \supset N_x(x,\epsilon_0)$ holds on the level of germs, then $S \supset N(x,\epsilon_0)$.
\item[iii)] The real-analytic set $N(x,\epsilon)$ has finitely many irreducible components.
\item[iv)] \cite[Proposition~III.5, p.~39]{Narasimhan} holds for $B(x,\epsilon_0)$, i.e.\
there is a non-zero analytic function $\delta_x$ on  $B(x,\epsilon_0)$, such that
the set of regular points $N_{reg}(x,\epsilon_0)$ contains $\{x: \delta_x(x) \neq 0\}$.
\end{itemize}
\end{lemma}

\begin{proof} Claim i) follows from \cite[Corollary V.1]{Narasimhan}. 
Claim ii) is a consequence of the Weierstrass preparation theorem, 
proven in \cite[Theorem~V.1, p.~98]{Narasimhan}. We have to shrink $\epsilon_0$
somewhat more to achieve this.
\par
The third
claim is obvious on the level of germs (\cite[Proposition~III.1, p.~32]{Narasimhan}), 
i.e.\ $N_x(x,\epsilon_0) = \cup_{i=1}^k N_{x,i}$ with $N_{x,i}$ irreducible germs. 
On $B(x,\epsilon_0)$ by ii)  there are irreducible analytic subsets $N_i$ with germs $N_{x,i}$.
Moreover, by ii) we have $N(x,\epsilon_0) = \cup_{i=1}^k N_i$. This proves iii). 
\par
The statement iv) is claimed in \cite[Proposition~III.5, p.~39]{Narasimhan}
for irreducible germs. It can obviously be extended to a finite number of 
irreducible components
taking the product of the corresponding functions $\delta_i$ on the
components $N_i$.
\end{proof}

\bold{Remark.} The containment in iv) may be strict; see
e.g. \cite[Example~3, page~106]{Narasimhan}.

\bold{Notation.} Let $\epsilon_0 = \epsilon_0(x)$ be as in
Lemma~\ref{lemma:N:stabilizes}. We denote $B(x,\epsilon_0)$ by $U(x)$
and  $N(x,\epsilon_0) \subset U(x)$ by $N(x)$.
\medskip

The following lemma follows directly from the definition of $N(x)$ and
Lemma~\ref{lemma:N:stabilizes}. 
\begin{lemma}
\label{lemma:zariski:density}
Suppose $f: U(x) \to \reals$ is a real analytic function such that
$f(y) = 0$ for $\nu$-almost-all $y \in N(x)$. Then $f$ is identically
zero on $N(x)$. 
\end{lemma}

%
%

\begin{corollary}
\label{cor:exist:regular:point}
Let $\epsilon_0$ and $B(x,\epsilon_0)$ be as in Lemma~\ref{lemma:N:stabilizes}. Then, 
\begin{displaymath}
\nu(N(x)  \cap \{\delta_x \neq 0\}) > 0, \quad \text{in particular} \quad
\nu(N(x)_{reg} ) > 0. 
\end{displaymath}
\end{corollary}

\begin{proof} Suppose not. Then, $\nu$ is supported on
 a proper real analytic subset of $N(x)$, which contradicts
Lemma~\ref{lemma:N:stabilizes} and the definition of $N(x)$. 
\end{proof}

\begin{lemma}
\label{lemma:N:SL2:invariant}
The sets $N(x)$, $N(x)_{reg}$ and $N(x)_{sing}$ are
$SL(2,\reals)$-equi\-variant in the following sense: 
suppose  that $g \in
SL(2,\reals)$, and let $U(g,x) = U(gx) \cap g U(x)$. Then $U(g,x)$ is
an open neighborhood of $gx$, and on $U(g,x)$ we have
\begin{displaymath}
N(g x) = g N(x), \quad N(g x)_{reg} = g N(x)_{reg}, \quad N(g x)_{sing} =
g N(x)_{sing}. 
\end{displaymath}
 \end{lemma}

\bold{Proof.} Since $\nu$ is $g$-invariant, we have $N(gx) = g N(x)$
on $U(g,x)$ by the definition of $N(x)$. 
Since the $SL(2,\reals)$ action is smooth, the same argument shows
that $N(gx)_{reg} = g N(x)_{reg}$. The final assertion follows from
the fact that $N(x)_{sing} = N(x) \setminus N(x)_{reg}$. \mc{check}
\qed\medskip

In view of Lemma~\ref{lemma:N:SL2:invariant}, the function 
$\dim N(x)$ is $SL(2,\reals)$-invariant and $\nu$-measurable on the
stratum. Therefore, since $\nu$ is assumed to be ergodic, there exists
a set $\Phi$ with $\nu(\Phi) = 1$ and $d \in \natls$ such that $\dim
N(x) = d$ for all $x \in \Phi$.

\begin{lemma}
\label{lemma:subset:full:dimension}
Suppose $N \subset U$ is the intersection of a real-analytic subset of
$\cx^n$ with an open set 
$U \subset \cx^n$,  $y \in N_{reg}$, $V \subset U$ is a neighborhood
of $y$ satisfying Lemma~\ref{lemma:N:stabilizes}, 
and $N'$ is a real-analytic set such that $N' \cap V \subset N
\cap V$ and $\dim N' = \dim N$. Then, $N' \cap V = N \cap V$. 
\end{lemma}

\begin{proof} On the level of germs at $y$ this is precisely \cite[Proposition~7, p.~41] {Narasimhan}. By our choice of neighborhoods according to 
Lemma~\ref{lemma:N:stabilizes} ii), an equality of the germs implies
equality of the analytic sets that induce the germs.
\end{proof}

\begin{proposition}
\label{prop:good:regular}
For $\nu$-almost all $y$ in the stratum, we have $y \in N(y)_{reg}$. 
\end{proposition}

\begin{proof} 
Let 
\begin{displaymath}
E = \{ y \st y \in N(y)_{reg} \}.
\end{displaymath}
By Lemma~\ref{lemma:N:SL2:invariant} the set $E$ is
$SL(2,\reals)$-invariant. Therefore, by ergodicity,  
it is enough to show that $\nu(E) > 0$. We start with some $y_0$
such that $N(y_0)$ is not empty.
\par
Choose an arbitrary $x \in N(y_0) \cap \{\delta_{y_0} \neq 0\} \cap \Phi$. 
By Corollary~\ref{cor:exist:regular:point}, for the
neighborhood $U(x)$ of $x$ we know that $\nu(N(x)_{reg}) > 0$. 
Therefore
\begin{equation}
\label{eq:mu:Nx:reg:cap:Ux:cap:Phi}
\nu(N(x)_{reg} \cap \Phi) > 0. 
\end{equation}
Suppose $y \in N(x) \cap \{\delta_x \neq 0\}  \cap \Phi$. Choose a neighborhood
$V$ of $y$ with $V \subset U(y) \cap U(x)$. Then, by
Lemma~\ref{lemma:N:stabilizes} i) we have $N(y) \cap V \subset N(x) \cap
V$. Also since
both $x \in \Phi$ and $y \in \Phi$, we have $\dim N(x) = \dim
N(y)$. Therefore by Lemma~\ref{lemma:subset:full:dimension}, we have
$N(y)\cap V = N(x) \cap V$. Hence we may take $\delta_x = \delta_y$
and conclude 
$$N(y)\cap \{\delta_y \neq 0\}  \cap V = N(x) \cap  \{\delta_x \neq 0\}\cap V$$
as well as $N(y)_{reg} \cap V = N(x)_{reg} \cap
V$ . Since $y$ was assumed to be in $N(x) \cap \{\delta_y \neq 0\}$, we have $y \in
N(y) \cap \{\delta_y \neq 0\}$. Thus, $y \in E$. We have shown that 
\begin{displaymath}
N(x) \cap \{\delta_x \neq 0\} \cap \Phi \subset E. 
\end{displaymath}
Therefore, by (\ref{eq:mu:Nx:reg:cap:Ux:cap:Phi}), $\nu(E) > 0$. 
\end{proof}

\begin{proposition}
\label{prop:N:linear}
For $\nu$-almost all $x$, $N(x) \subset U(x)$ is affine. 
In particular, $TN(x)$ is preserved by the complex structure $J$. 
\end{proposition}

\begin{proof}[{Outline of Proof}] The relevant parts of the argument
  labelled ``Proof of
  Propositions 4.1 and 2.2'' in \cite[\S{4}]{Avila:Gouezel} apply.  
The ``second step'' and ``third step'' can be done for almost all $y
\in N(x)$. Then by Lemma~\ref{lemma:zariski:density}, the conclusion
of the ``third step'' holds for all $y \in N(x)$. Then the ``fourth
step'' of the ``Proof of
  Propositions 4.1 and 2.2'' proceeds exactly as in \cite[\S{4}]{Avila:Gouezel}, and this shows that
$N(x)$ is affine.
\end{proof}

The above argument also proves the following:
Let $Y$ be a real-analytic space supporting an $SL(2,\reals)$-action. 
Suppose $F: \cN \to Y$ is a (locally defined) real-analytic function,
in the sense that for each neighborhood $N(x)$, the restriction of $F$
to $N(x)$ is real-analytic.  Such a function $F$ is
$SL(2,\reals)$-equivariant, if it invariant in in the sense of
Lemma~\ref{lemma:N:SL2:invariant}, i.e. that for $g \in SL(2,\reals)$
and $x$ in the support of $\nu$, for $y \in N(gx) \cap g N(x)$ we have
$F(gy) = g F(y)$. Finally, we let 
$$\cL(x) = \{ y \in N(x) \st F(y) = F(x) \}$$ 
be the level set of $F$. 
\par
\begin{proposition}
\label{prop:cL:linear}  Suppose  $F: \cN \to Y$ is a real-analytic 
$SL(2,\reals)$-equivariant function. In addition, suppose that for 
$\nu$-almost all $y \in N(x)$, $y$ is a
regular point of $F$. 
Then, for
$\nu$-almost all $x$, the level set $\cL(x)$ through $x$ 
is affine, and $T\cL(x)$ is preserved by the complex structure $J$.  
\end{proposition}

\section{The Forni subspace revisited.}
\label{sec:forni:revisited}
Everything in this 
section is a local statement about the intersection of the tangent
space to the real analytic envelope of $\nu$ and the Forni subspace. 
Since everything is local (say around $x$) we may assume by Proposition~\ref{prop:N:linear}
that the real analytic envelope $N=N(x)$ of $\nu$ is affine.

Since the form $B_x^\reals(\cdot, \cdot)$ depends real-analytically on
$x$, we may (for $\nu$-almost all $x \in \cH_1(\kappa)$)
shrink $U(x)$ so that $\dim F(y)$ stays constant for all
$y \in N(x)$. Then $F(y)$ depends real-analytically on $y \in N(x)$.

We have the following:
\begin{lemma}
\label{lemma:properties:Forni}
For $\nu$-almost all $x$ there exists a neighborhood $U(x)$ such that
for all $y \in N(x)$ the following hold:
\begin{itemize}
\item[{\rm (a)}] The subspace $F^\perp(y)$ defined as the orthogonal
  complement of $F(y)$ using the Hodge inner product is
  $SL(2,\reals)$-invariant.
\item[{\rm (b)}] For $v \in F(y)$, and $w \in F^\perp(y)$, $\langle v
  , w \rangle = 0$, where $\langle \cdot, \cdot \rangle$ denotes the
  symplectic form. 
\item[{\rm (c)}] If $y = a+bi$, then the space spanned by $a$ and $b$ is
  contained in $F^\perp(y)$.
\item[{\rm (d)}] The restriction of the symplectic form to $F(y)$ is
  non-degenerate.
\end{itemize}
\end{lemma}

\begin{proof} Since $F(y)$ depends real-analytically on $y$ and the Hodge
inner product does, also  $F^\perp(y)$ depends
real-analytically on $y$. Therefore, by
Lemma~\ref{lemma:zariski:density} it is enough to show (a)-(d) for $y$
in the support of $\nu$. The statements (a), (b), (d) follow
immediately from Theorem~\ref{theorem:properties:forni}. To prove (c),
we claim that the tautological subbundle (the one spanned by $a$ and
$b$) is symplectically orthogonal to $F$ on the support of $\nu$. 

Indeed, $a$ spans the Lyapunov subspace $E_{-1}(y)$ of the cocycle $A(y,t)$
corresponding to the Lyapunov exponent $-1$. Since $A(y,t)$ preserves
the symplectic structure, and the Hodge norm of any vector $v \in F$
is preserved by the cocycle, one gets
\begin{displaymath}
\langle v, a\rangle = \langle A(y,t) v, A(y,t) a \rangle \to 0
\end{displaymath}
as $t \to \infty$. Therefore $\langle v, a\rangle = 0$. By a similar
argument involving taking a limit as $t \to -\infty$,  
we get $\langle v, b\rangle = 0$. 
\end{proof}

\section{A connection on the unstable leaf.}
\label{sec:connection}

\bold{Period Coordinates and the geodesic flow.}
Locally around a point $x=(M,\omega) \in \cH_1(\kappa)$ 
we identify the tangent space to the stratum with 
$$H^1(M,\Sigma,\cx) = H^1(M,\Sigma,\reals) \otimes_\reals \cx. $$
In this identification $SL(2,\reals)$ acts on $\cx \cong \reals^2$.
Let $g_t = \left(\begin{smallmatrix} e^t & 0 \\ 0 & e^{-t} \\ \end{smallmatrix}
\right)$ denote the geodesic flow.

The condition $area(M)=1$ defining the ``unit hyperboloid''
$\cH_1(\kappa)$ in cohomological coordinates
is equivalent to $\langle p([\omega]),p([\bar\omega])\rangle=-2i$,
where $\langle\cdot,\cdot\rangle$ is the symplectic product, and
$p([\omega])\in H^1(M;\cx)$ is the absolute cohomology class of the
holomorphic 1-form defining the flat structure. Differentiating,
we get the following equation for the vectors $\alpha$ in
the tangent space considered as a real
hyperplane in $H^1(M,\Sigma;\cx)$:
$$
\langle p(\alpha),\overline{p([\omega])}\rangle +
\langle p([\omega]),\overline{p(\alpha)}\rangle = 0\,,
$$
where  $p:H^1(M,\Sigma;\cx)\to  H^1(M;\cx)$  is the natural projection.
Thus,  when  $c:=p(\alpha)\in  H^1(M;\reals)$  is  a purely real class and
$p[\omega]=a+ib$,
we  get  the condition $\langle c,b\rangle=0$. Similarly,
when  $i\cdot c:=p(\alpha)\in  i\cdot H^1(M;\reals)$  is  a purely imaginary
class, we  get  the condition $\langle c,a\rangle=0$.
Finally, the vector $[\bar\omega]$ is tangent to the Teichm\"uller geodesic
and is transversal both to $p^{-1}(\operatorname{Ann}(b))$ and
to $i\cdot p^{-1}(\operatorname{Ann}(a))$,
where by $\operatorname{Ann}(v)$ we denote
the symplectic annihilator of a real absolute cohomology class
$v$ in the ambient space $H^1(M;\reals)$. Hence, the sum
\begin{displaymath}
p^{-1}(\operatorname{Ann}(b))\oplus\reals[\bar\omega]
\oplus i \cdot p^{-1}(\operatorname{Ann}(a))
\end{displaymath}
is  a  direct  sum and thus defines a hyperplane. We have proved that
all vectors in this hyperplane belong to the tangent space to the
``unit  hyperboloid''.  Hence,  this hyperplane coincides with the
tangent space to the unit hyperboloid.

By \cite[\S{2}]{Forni:Deviation}, we may consider the subspace
$p^{-1}(\operatorname{Ann}(b))$ as strictly unstable; the subspace
$\reals[\bar\omega]$ as neutral and the subspace $i
p^{-1}(\cdot\operatorname{Ann}(a))$ as strictly stable for the
Teichm\"uller geodesic flow. We identify leaves of the unstable
foliation of $g_t$ with afine subspaces of $H^1(M, \Sigma, \reals)$. More
precisely, for $x = a + i b$, let 
\begin{displaymath}
W^{uu}(x) = TN(x) \cap \{ s \in H^1(M,\Sigma,\reals) \st \langle p(s),
p(b) \rangle = 0 \},
\end{displaymath}
and
\begin{displaymath}
W^{ss}(x) = TN(x) \cap i \; \{ s \in H^1(M,\Sigma,\reals) \st \langle
p(s), p(a) \rangle = 0 \},
\end{displaymath}
where $N(x)$ is the affine subspace defined in
\S\ref{sec:real:analytic:envelope}, and $TN(x)$ is its linear part. 
(Note that if the support of $\nu$ is an affine manifold $\cM$, then
the subspace $N(x) \subset H^1(M,\Sigma,\cx)$ 
coincides with the tangent space to $\cM$ at $x$). In view of the
above discussion, 
if $x = a + b i$, then the unstable leaf through $x$ is
identified with $x+ W^{uu}(x)$, and the stable leaf through $x$ is
identified with $x+ W^{ss}(x)$. 

\par
Let $A(x,t)$ denote the Kontsevich-Zorich cocycle (i.e. the action of
the Gauss-Manin connection on the real relative cohomology). 
We view $A(x,t)$ as an endomorphism
of  $H^1(M, \Sigma, \reals)$, 
%
whereas \cite{Avila:Gouezel} use the parallel transport 
$Dg_t: H^1(M, \Sigma, \reals) \to H^1(g_tM, \Sigma, \reals)$. 
The two maps differ by a scaling factor $e^t$.


\bold{Lyapunov Exponents.}
Let $\{ \lambda_i \st i \in \Lambda\}$ denote the Lyapunov
exponents of the cocycle $A(x,t)$ acting on the real Hodge bundle
$H^1_{\reals}$  with  respect to the Teichm\"uller flow on the support of
the   ergodic   $SL(2,\reals)$-invariant   probability  measure  $\nu$.
Let
\begin{displaymath}
\Lambda^- = \{ i \in \Lambda \st \lambda_i < 0 \}, \qquad \Lambda^+ =
\{ i \in \Lambda \st \lambda_i > 0 \}. 
\end{displaymath}
For a generic point $x$ in the support of $\nu$, let $E_i(x)$ be the Lyapunov
subspace (at $x$) corresponding to the Lyapunov exponent
$\lambda_i$. Since we are considering the invertible case of the
Osceledets Multiplicative ergodic theorem, we have
\begin{displaymath}
H^1(M,\reals) = \bigoplus_{i \in \Lambda} E_i(x).
\end{displaymath}

Suppose $x \in \operatorname{supp} \nu$, \mc{fixme}
and $s \in W^{uu}(x)$.  For $s$ sufficiently
small $x+s$ is a well-defined point in $\cH_1(\kappa)$ and 
is in the same leaf of the unstable foliation as $x$. We do not assert
that $x+s$ is in the support of $\nu$. In our local coordinates,
\begin{displaymath}
g_t(x + s) = g_t x + e^{t} A(x,t) s.
\end{displaymath}

\bold{A flat connection on the unstable leaf.}
Suppose $v(x) \in F(x)$. 
Since for small $s \in W^{uu}(x)$ the subspaces
$F(x+s)$ and $F^\perp(x)$ are complementary, we
may write
\begin{displaymath}
v(x) = v_x(x+s) - u(x,s), \ \text{ where $v_x(x+s) \in F(x+s)$,
  $u(x,s) \in F^\perp(x)$. }
\end{displaymath}
Then,
\begin{equation}
\label{eq:vxs}
v_x(x+s) = v(x) + u(x,s).
\end{equation}
Thus we have a linear map $P^+(x,x+s): F(x) \to F(x+s)$ such that
\begin{displaymath}
P^+(x,x+s) v(x) = v_x(x+s). 
\end{displaymath}
In view of Lemma~\ref{lemma:properties:vx} below, it is easy to see
that the map $P^+(x,x+s)$ coincides with the restriction to the Forni
subspace of the measurable flat connection $P^+$ defined in
\cite[4.2]{Eskin:Mirzakhani:measures}. The main difference is that in
our context, the map $P^+(x,x+s)$ depends real-analytically on $x$ and
$x+s$.

Fix $\beta > 0$ and let $K_\beta'$ be the set where all saddle
connections have length at least $\beta$. 
Then the Hodge inner product $Q_x( \cdot, \cdot)$ is uniformly 
continuous on $K_{\beta/2}'$.
We can pick a compact subset $K_\beta
\subset K_\beta'$ of positive measure such that the
subspaces $F(x)$ are uniformly continuous on $K_\beta$. We denote by
$\|\cdot\|$ the Hodge norm, given by $\|v\|_x = Q_x(v,v)^{1/2}$. 

\begin{lemma}
\label{lemma:properties:vx}
Suppose $x$ is $\nu$-generic. \mc{specify what this means} Then,
for $s \in W^{uu}(x)$,
\begin{itemize}
\item[{\rm (a)}] The vector $P^+(x,x+s)v(x)$ depends real-analytically
  on $s$. 
\item[{\rm (b)}] $P^+(x,x+s)v(x)$ is the unique vector in $F(x+s)$ such that
  for any sequence $t_n \to -\infty$ with $g_{t_n} x \in K_\beta$, we have 
\begin{equation}
\label{eq:going:to:zero}
\|A(x,t_n) P^+(x,x+s)v(x) - A(x,t_n) v(x) \| \to 0. 
\end{equation}
\item[{\rm (c)}] For any $s_1$ and $s_2$ in $W^{uu}(x)$, we have
\begin{displaymath}
P^+(x+s_1,x+s_1+s_2) P^+(x,x+s_1) v(x) = P^+(x,x+s_1+s_2) v(x).
\end{displaymath}
\end{itemize}
\end{lemma}
\begin{proof} The statement (a) is clear from the definition of
$P^+(x,x+s)$ and the fact that $F(x+s)$ is analytic in $s$. To see (b),
we will apply the geodesic flow to (\ref{eq:vxs}). 
Let $y_n = g_{t_n} x$. Let $\xi_n \in W^{uu}(y_n)$  
be such that
$g_{t_n} (x+s) = y_n + \xi_n$. 
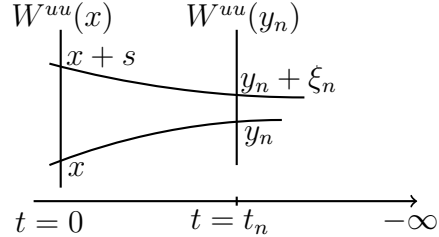
\begin{center}
\begin{tikzpicture}[scale=0.3]

\draw[thick, ->] (0,0.4)--(17,0.4);
\draw[thick] (9,0.2)--(9,0.6);
\draw[thick] (1.2,1)--(1.2,8);
\draw[thick] (9,2)--(9,8);
\draw[thick] (11,4) parabola (0.7,2);
\draw[thick] (12,5) parabola (0.7,6.5);

\node (x) at (0.7,-0.5) {$t=0$};
\node (x) at (8.7,-0.5) {$t=t_n$};
\node (x) at (16.7,-0.5) {$-\infty$};
\node (x) at (1.3,8.5) {$W^{uu}(x)$};
\node (x) at (9.3,8.5) {$W^{uu}(y_n)$};
\node (x) at (3,6.8) {$x+s$};
\node (x) at (11.3,5.8) {$y_n+\xi_n$};
\node (x) at (1.9,1.8) {$x$};
\node (x) at (10,3.3) {$y_n$};
\end{tikzpicture}
\end{center}

We have $\xi_n \to 0$ as $t_n \to -\infty$. Also,
\begin{displaymath}
A(x,t_n) v(x) = A(x,t_n) P^+(x,x+s)v(x) - A(x,t_n) u(x,s). 
\end{displaymath}

We may start with $v(x)$ of norm $\|v(x)\|=1$, and let 
$$\alpha = \|P^+(x,x+s) v(x)\|.$$
Note that for $s \in W^{uu}(x)$ sufficiently small
and $t < 0$, 
\begin{displaymath}
A(x,t) = A(x+s,t)
\end{displaymath}
Then, 
by our convention on the Kontsevich-Zorich cocycle
\begin{equation}
\label{eq:norm:one}
\|A(x,t_n) v(x) \|_{y_n} = 1, \quad \|A(x,t_n) P^+(x,x+s)v(x)\|_{y_n+\xi_n} = \alpha.
\end{equation}
Also,  
$A(x,t_n) v(x) \in F(y_n)$ and
\begin{displaymath}
A(x,t_n) P^+(x,x+s)v(x) \in F(y_n +\xi_n), 
\quad A(x,t_n) u(x,s) \in F^\perp(y_n).  
\end{displaymath}
Since $\xi_n \to 0$ and $y_n \in K_\beta$, we have $F(y_n + \xi_n) \to
F(y_n)$. This, combined with (\ref{eq:norm:one}) implies that
\begin{equation}
\label{eq:A:wx}
A(x,t_n) u(x,s) \to 0,
\end{equation}
which proves (\ref{eq:going:to:zero}). 
To show the uniqueness in (b),
note that since the action of the
Kontsevich-Zorich cocycle on the Forni subspace is isometric, there
can be at most one vector $v' \in F(x+s)$ such that
\begin{equation}
\label{eq:unique:forni}
\|A(x,t_n) v' - A(x,t_n) v \| \to 0. 
\end{equation}
This completes the proof of (b). 

To see (c), note that (\ref{eq:unique:forni}) holds with $v' \in
F(x+s_1+s_2)$ being
either the left-hand-side or the right-hand-side of the displayed equation in
(c). Since the $v'$ satisfying (\ref{eq:unique:forni}) is unique, 
the two sides of the displayed equation in (c) are equal. 
\end{proof}

\bold{Notation.} 
In view of (c), when there is no potential for confusion, 
we denote $v_x(x+s) = P^+(x,x+s) v(x)$ simply by $v(x+s)$. 

\bold{Remark.} 
The equation (\ref{eq:A:wx}) implies that
\begin{equation}
\label{eq:vxs:components}
v(x+s) = v(x) + v_0(x,s) + \sum_{i \in \Lambda^+} v_i(x,s)
\end{equation}
where $v_0(x,s) \in E_0(x) \cap F^\perp(x)$, and for $i \in \Lambda^+$, 
$v_i(x,s) \in E_i(x) \subset F^\perp(x)$. 
The subspace $E_0(x) \cap
F^\perp(x)$ may be non-empty since there can be zero Lyapunov
exponents outside of the isometric subbundle $F$, see
\cite{Forni:Matheus:Zorich:Zero}. 
 \mc{give exact reference.}


\begin{lemma}{\rm (cf.\ \cite[Proposition 4.4(b)]{Eskin:Mirzakhani:measures})}
\label{lemma:parallel:transport:hodge:inv}
The fiberwise parallel transport $P^+$ 
defined above preserves the Hodge inner product
$Q(\cdot, \cdot)$ on $F$. In other words, for $v, w \in F(x)$, and $s
\in W^{uu}(x)$, 
\begin{displaymath}
Q_{x+s}(v(x+s),w(x+s)) = Q_x(v(x), w(x))\, .
\end{displaymath}
\end{lemma}

\begin{proof} 
Let $t_n$, $y_n$, $\xi_n$ be as in the proof of
Lemma~\ref{lemma:properties:vx}. We have
\begin{equation}
\label{eq:Qyn:v1}
Q_{y_n}(A(x,t_n) v(x), A(x,t_n)w(x)) = Q_x(v(x),w(x))
\end{equation}
and
\begin{equation}
\label{eq:Qyn:v2}
Q_{y_n+\xi_n}(A(x,t_n) v(x+s), A(x,t_n)w(x+s)) = Q_{x+s}(v(x+s),w(x+s))\,.
\end{equation}
We have $\xi_n \to 0$, and by Lemma~\ref{lemma:properties:vx} (b) we have
\begin{equation*}
\begin{aligned}
\|A(x,t_n) v(x+s) - A(x,t_n) v(x) \| &\to 0, \\
\|A(x,t_n) w(x+s) - A(x,t_n) w(x) \| &\to 0. 
\end{aligned}
\end{equation*}
Thus, the left-hand-sides of (\ref{eq:Qyn:v1}) and (\ref{eq:Qyn:v2})
approach each other as $t_n \to \infty$. Thus the right-hand-sides of
(\ref{eq:Qyn:v1}) and (\ref{eq:Qyn:v2}) are equal. 
\end{proof}

\section{A Formula for $P^+(x,x+s)$.}
\label{sec:preliminary:formula}
In this section, we derive an explicit formula,
Lemma~\ref{lemma:formula:v},  for the fiberwise parallel
transport map $P^+$ defined in the previous section. The main tool is
the real-analyticity of the fiberwise connection $P^+$. 
\medskip

Let $A_i(x,t)$ denote the restriction of $A(x,t)$ to $E_i(x)$. 
\begin{lemma}
\label{lemma:taylor:expand}
Let $s$, $v_0(x,s)$ and $v_i(x,s)$ be as in (\ref{eq:vxs:components}). 
For every $\delta > 0$ there exists a compact set $\hat{K}=
\hat{K}(\delta)$ of measure
at least $1-\delta$, such that the following holds:
Suppose $t
< 0$ is such that $g_t x \in \hat{K}$. Then for $i \in \Lambda^+ \cup
\{0\}$ \mc{fix notation} we have the Taylor expansion 
\begin{equation}
\label{eq:taylor}
v_i(x,s) = A_i(x,t)^{-1} \sum_{\alpha} c_{i,\alpha} (e^{t} A(x,t)
s)^{\alpha}, 
\end{equation}
where $\alpha$ is a multi-index, 
and the $c_{i,\alpha} \in E_i(g_t x) \cap F^\perp( g_t x)$ are
bounded independently of $t$. 
\end{lemma}
\begin{proof} 
Let $y = g_t x$, and let $\xi$ be such that 
\begin{displaymath}
g_t(x+s) = y+ \xi. 
\end{displaymath}
Then, $\xi =  e^{t} A(x,t) s$. Write
\begin{displaymath}
A(x,t) v(x+s) = w(y+\xi).
\end{displaymath}
Since $\|v(x)\|=1$, and $A(x,t)$ acts isometrically on the Forni
subspace, we deduce $\|w(y)\| = 1$. 

Since $w(y+\xi)$ depends real-analytically on $\xi$, we can Taylor expand
\begin{displaymath}
w_i(y+\xi) = w_i(y) + \sum_{\alpha} c_{i,\alpha} \xi^{\alpha},
\end{displaymath}
where $c_{i,\alpha} \in E_i(y)$, $\alpha$ is a multi-index and 
we use the standard multi-index notation. In particular $|\alpha|$
denotes the sum of the indices in~ $\alpha$. 

Let $K_\beta$ be as in \S\ref{sec:connection}. 
We choose $K' \subset K_\beta$ of measure at least $1-\delta/2$
so all the Lyapunov subspaces $E_i(y)$ are uniformly continuous functions of $y
\in K'$. Also,
for each $i,\alpha$, choose a number $\epsilon_{i,\alpha}$ such that 
$\sum_{i,\alpha} \epsilon_{i,\alpha} < \delta/2$. 
We choose a compact set $K_{i,\alpha}$ of measure at least
$1-\epsilon_{i,\alpha}$ such that the Taylor coefficient 
$c_{i,\alpha}$ is uniformly bounded (in terms of $i$ and $\alpha$) on
$K_{i,\alpha}$. Finally, let 
\begin{displaymath}
\hat{K} = K' \cap \bigcap_{i,\alpha} K_{i,\alpha}. 
\end{displaymath}
Then the measure of $\hat{K}$ is at least $1-\delta$, and 
also all the Taylor coefficients $c_{i,\alpha}$ are uniformly bounded
(in terms of $i$ and $\alpha$) for $y \in \hat{K}$. 

For $i \in \Lambda^+
\cup \{0 \}$, \mc{fix notation here also}
\begin{displaymath}
A_i(x,t) v_i(x,s) = \sum_{\alpha} c_{i,\alpha} (e^{t} A(x,t) s)^{\alpha}, 
\end{displaymath}
(and we used the fact that $v_i(0) = 0$). Now applying $A_i(x,t)^{-1}$
to both sides, we obtain (\ref{eq:taylor}). 
\end{proof}

\bold{Remark.}  It is clear from Lemma~\ref{lemma:taylor:expand} 
that $v_i(x,s)$ is a polynomial in~$s$, i.e.\ for
each $s \in W^{uu}(x)$ the dependence of $v_i(x, s)$ for varying $s$ 
in a fixed basis of $E_i$ is polynomial in $s$. Indeed if $|\alpha|$ is
sufficiently large (depending on the Lyapunov spectrum) then for any
$t < 0$
such that $g_t x \in \hat{K}$, the coefficient of $s^\alpha$ in
$v_i(x,s)$ is
bounded by
\begin{displaymath}
\|A_i(x,t)^{-1}\| \|e^{t} A(x,t)\|^{\alpha} \le C 
\|A_i(x,t)^{-1}\| e^{(1-\lambda_2)|\alpha|t } \le C e^{-t}
e^{(1-\lambda_2)|\alpha| t}
\end{displaymath}
where $\lambda_2 < 1$ is second Lyapunov exponent of the
Kontsevich-Zorich cocycle. Thus, if $|\alpha| > 1/(1-\lambda_2)$, the
right-hand-side tends to $0$ as $t \to -\infty$.

Let $p: H^1(M, \Sigma, \reals) \to H^1(M, \reals)$ be the
natural projection map. We recall that the Lyapunov spectrum of the
Kontsevich-Zorich cocycle on $H^1(M, \Sigma, \reals)$ consists of
the Lyapunov spectrum of the Kontsevich-Zorich cocycle on $H^1(M,
\reals)$ union $\dim (\ker p)$ zeroes. It can also be shown that
if $\lambda$ is a nonzero Lyapunov exponent, $\tilde{E}_\lambda$ the
Lyapunov subspace of $H^1(M, \Sigma, \reals)$ corresponding to
$\lambda$, and $E_\lambda$ the Lyapunov subspace corresponding to
$\lambda$ on $H^1(M, \reals)$, then $p(\tilde{E}_\lambda) =
E_\lambda$, and $p$ induces an isomorphism between these two subspaces. 
For proofs of these statements, see e.g.\ \cite{Zorich}.

\begin{lemma}
\label{lemma:independent:negative}
Suppose $s \in W^{uu}(x) \cap \bigoplus_{j \in \Lambda^+} \tilde{E}_j(x)$. 
Then, $F(x+s) = F(x)$. 
\end{lemma}

\begin{proof} Suppose $v(x) \in F(x)$, and for $i \in \Lambda^+ \cup
\{0\}$ let $v_i(s)$ be as in (\ref{eq:vxs:components}). 
Write
\begin{displaymath}
s = \sum_{j \in \Lambda^+} s_j,
\end{displaymath}
where $s_j \in \tilde{E}_j(x)$. 
Then, \mc{explain notation} as $t \to -\infty$, 
\begin{displaymath}
\| e^{t} A(x,t) s_j \| \approx e^{(1+\lambda_j) t} \|s_j\| 
\end{displaymath}
and 
\begin{displaymath}
\|A(x,t)^{-1} c_{i,\alpha} \| \le C e^{-\lambda_i t}, \quad \text{ where }
\lambda_i \le 1. 
\end{displaymath}
It now follows from (\ref{eq:taylor}) that for $t < 0$ such that $g_t x
\in \hat{K}$, 
\begin{displaymath}
\|v_i(x,s)\| \le C \|A(x,t)^{-1}\|  \, \max_{j \in \Lambda^+} \|e^{t}
A(x,t) s_j \| \le C \max_{j \in \Lambda^+}  \,
e^{(1-\lambda_i+\lambda_j) t}  \|s_j\|
\end{displaymath}
which tends to $0$ as $t > -\infty$ 
since $1-\lambda_i \ge 0$ and $\lambda_j > 0$. 
Thus, $v(x+s)$ is independent of the $s_j$ for $j \in \Lambda^+$. 
\end{proof}

Let $N(x)$ be as in \S\ref{sec:real:analytic:envelope}. By
Proposition~\ref{prop:N:linear}, we have
\begin{displaymath}
T N(x) = \cx \tensor T_\reals N(x), 
\end{displaymath}
where $T_\reals N(x)$ is a subspace of $H^1(M, \Sigma, \reals)$. 

Let $L(x) \subset {T_\reals N(x)}$
be the smallest $SL(2,\reals)$-invariant subspace such that $p(L(x))
\supset F^\perp(x) \cap T_\reals N(x)$. We have
\begin{equation}
\label{eq:L:plus:F:plus:kerp}
T_\reals N(x) = L(x) + p^{-1}(F(x))
\end{equation}
This sum of subspaces need not be direct. 

\begin{lemma}
\label{lemma:v:independent:F:perp}
Suppose $x_0 \in \operatorname{supp} \nu$, \mc{fixme} and $x \in
N(x_0)$. Then, for all $s \in L(x)$,  
\begin{displaymath}
F(x+s)=F(x). 
\end{displaymath}
\end{lemma}

\begin{proof} 
For $x \in N(x_0)$, let $\cL(x)$ denote the set of $y$
in $N(x_0)$ such that $F(y) = F(x)$. Then $\cL(x)$ is a
real-analytic set which is $SL(2,\reals)$-equivariant. We can think of
$F$ as a function from $N$ to the Grassmanian. 
Let $r$ be such that $DF$ has generically
rank $r$ on the support of $\nu$. Then, by
Lemma~\ref{lemma:zariski:density}, $\operatorname{rank} DF \le r$ on
$N(x_0)$ (otherwise 
the support of $\nu$ would be contained in a proper real-analytic
subset of $N(x_0)$), and by continuity, 
$\operatorname{rank} DF \ge r$ on a neighborhood of $x \in N$. 
By the implicit function theorem, $\cL(x)$ is regular 
on a neighborhood of $x$ in the support of $\nu$. Then, by
Proposition~\ref{prop:cL:linear}, for almost all $x$, 
the set $\cL(x)$ is affine and defined over $\reals$. Since being affine and
being defined over $\reals$ are real-analytically closed conditions,
by Lemma~\ref{lemma:zariski:density}, the same is true for all $x \in
N(x_0)$. 

Thus  there  exists  an  $SL(2,\reals)$-invariant  
subspace $\hat L(x) \subset T_\reals N(x)$
such that $\cL(x) = \{ x + z \st z \in \cx
\tensor   \hat   L(x)\}$.  Note that for $x = a + i b$, 
we know that $a \in
\tilde{E}_{-1}(x)$, therefore
\begin{displaymath}
W^{uu}(x) \cap \bigoplus_{j \in \Lambda^+} \tilde{E}_j(x) = T_\reals
N(x) \cap \bigoplus_{j \in \Lambda^+} \tilde{E}_j(x),
\end{displaymath}
Therefore,  by Lemma~\ref{lemma:independent:negative},
\begin{displaymath}
T_\reals N(x) \cap \bigoplus_{j \in \Lambda^+} \tilde{E}_j(x)
\subset \hat{L}(x),
\end{displaymath}
hence 
\begin{displaymath}
 p(T_\reals N(x)) 
\cap \bigoplus_{i  \in  \Lambda^{+}}  E_i(x) \subset p(\hat  L(x)).  
\end{displaymath}
Let us now prove that
\begin{equation}
\label{eq:contained:in:P:hat:L}
p(T_\reals N(x)) \cap F^\perp(x)\subset p(\hat L(x)).
\end{equation}
Since  $p(\hat  L(x))$  is  $SL(2,\reals)$-invariant,  
$V(x)=p(\hat L(x)) \cap
F^\perp(x)$  is  an  $SL(2,\reals)$-invariant subbundle. 
Clearly, 
\begin{equation}
\label{eq:conained:in:V}
p(T_\reals N(x)) \cap \bigoplus_{i \in \Lambda^{+}} E_i(x) \subset V(x).
\end{equation}
By Theorem~\ref{theorem:properties:forni} (b)
the  subbundle  $V(x)$ is symplectic. Then, its
symplectic   complement   $V'$   inside   $F^\perp \cap
p(T_\reals N)$  is  also
$SL(2,\reals)$-invariant and $p(T_\reals N) \cap F^\perp =
V\oplus V'$. If $V'\neq\{0\}$,
then  by Theorem~\ref{theorem:properties:forni} (b), 
the Lyapunov spectrum of $V' \subset p(T_\reals N)$ 
contains at
least   one   nonzero  Lyapunov  exponent  and  hence,  since  it  is
symplectic,  at  least  one strictly positive Lyapunov exponent. This
contradicts (\ref{eq:conained:in:V}). Thus 
(\ref{eq:contained:in:P:hat:L}) holds. 

Since  $L(x)$ is the minimal
$SL(2,\reals)$-invariant   subbundle   such that $p(L(x)) \supset
p(T_\reals N(x)) \cap F^\perp(x)$, we  have, by
(\ref{eq:contained:in:P:hat:L}), 
$L(x)\subseteq\hat L(x)$. From the definition of $\hat{L}(x)$, this
implies that $F(y) = F(x)$ for $s \in L(x)$.
\end{proof}

\begin{lemma}
\label{lemma:formula:v}
Suppose $x \in N(x_0)$ is in the support of the measure $\nu$. \mc{fixme}
If $x = a+b\bi$ (with $Area(x) = \langle a, b \rangle = 1$), we have
for any $v(x) \in F(x)$ and any $s \in W^{uu}(x)$,
\begin{equation}
\label{eq:formula:v}
P^+(x,x+s)v(x) \equiv v(x+s) = v(x) + \langle v, p(s)\rangle p(b).
\end{equation}
\end{lemma}

\begin{proof} 

For $s$ sufficiently small, we have in view of
(\ref{eq:L:plus:F:plus:kerp}), 
\begin{displaymath}
T_\reals N(x) = L(x+s) + p^{-1}(F(x)).
\end{displaymath}
Therefore, we may write $s = s_1 + s_2$, where $s_1 \in L(x+s)$ and
$s_2 \in p^{-1}(F(x))$. 
Then, by
Lemma~\ref{lemma:v:independent:F:perp}, $F(x+s) = F(x+s_2)$. Then, by
the definition of $v$, \mc{make more precise}
\begin{displaymath}
v(x+s) = v(x+s_2).  
\end{displaymath}
We now apply (\ref{eq:taylor}) with $s = s_2$. 
Note that $A(x,t)$ acts with zero Lyapunov exponents on
$p^{-1}(F(x))$.  Let $v_i(x,s_2)$ be as in (\ref{eq:vxs:components}). Then, 
\begin{displaymath}
\|v_i(x,s_2)\| \le C \sum_\alpha \|A_i(x,t)^{-1} \|
e^{t|\alpha|}s_2^{\alpha} \le C 
\sum_{\alpha} e^{(-\lambda_i+
  |\alpha|)t} s_2^\alpha
\end{displaymath}
Since $\lambda_i \le 1$, we see that the coefficient of $s_2^\alpha$
for any $|\alpha| > 1$ tends to  
$0$ as $t \to -\infty$ with $g_t x \in \hat{K}$. Therefore all
quadratic and higher order terms vanish. Also the only surviving linear
term is with $\lambda_i = 1$. 
Note that the Lyapunov subspace on $H^1(M, \reals)$ 
corresponding to the Lyapunov exponent $1$ is $1$-dimensional, and
$p(b)$ belongs to it. 
Hence, 
\begin{displaymath}
v((x+s_1)+s_2) = v(x+s_1) + \Phi_{x}(s_2) p(b), 
\end{displaymath}
where $\Phi_{x}$ is a linear map. Thus, 
\begin{displaymath}
v(x+s) = v(x) + \Phi_{x}(s_2) p(b).
\end{displaymath}
But, we know that $v(x+s)$ has to be symplectically orthogonal to the
$SL(2,\reals)$ orbit at $x+s$, i.e. the space spanned by $p(a+s)$ and
$p(b)$. Thus, 
\begin{displaymath}
0 = \langle p(a+s), v(x+s) \rangle = \langle p(a+s), v(x) \rangle +
\Phi_{x}(s_2) \langle p(a+s), p(b) \rangle. 
\end{displaymath}
From $s \in W^{uu}(x)$ we deduce 
$\langle p(s), p(b) \rangle = 0$. 
Moreover, we have $\langle p(a), v(x) \rangle = 0$ and 
$\langle p(a), p(b) \rangle = 1$. Then,
\begin{displaymath}
\Phi_{x}(s_2) = - \langle p(s),v(x)\rangle = - \langle  p(s_2), v(x)\rangle. 
\end{displaymath}
Thus (\ref{eq:formula:v}) follows. 
\end{proof}

\section{The curvature of the connection.}
\label{sec:curvature}
We have defined a connection $P^+$ on unstable leaves. Similarly,
there is a connection $P^-$ on stable leaves. Also, by its definition,
the Forni subspace is constant along the neutral (i.e. geodesic flow)
direction. Since all of these 
connections are real-analytic, they can together define a connection
$P$ on the real-analytic envelope of the support of the
measure. Even though both $P^+$ and $P^-$ are flat
  virtually by definition, it is not obvious whether the connection
  $P$ is flat. The key calculation
of this paper is the following proposition, which can be viewed as
computing the curvature of $P$.  

\begin{proposition}
\label{prop:no:yeti}
Suppose $x_0$ is in the support of $\nu$, and let $N =
N(x_0)$. Let $TN$  
be the tangent space to the affine manifold $N$. We have by
Proposition~\ref{prop:N:linear}, that $TN = \cx \tensor T_\reals N$
for some subspace $T_\reals N \subset H^1(M, \Sigma, \reals)$. 
Then, for
all $x \in N$, 
\begin{equation}
\label{eq:no:yeti}
p(T_\reals N) \subset F^\perp(x)
\end{equation}
and $F(x)$ is locally constant on $N$. 
\end{proposition}

\begin{proof} 
Suppose (\ref{eq:no:yeti}) fails on a set of positive measure. 
Then there is a positive measure set of pairs $x',y$ in the support of $\nu$ such that such that $y-x' \in TN$ and $p(y-x') \not\in \cx \tensor
F^\perp(x')$. 

Write $x' = a' + b' \bi$. 
Let $H(x') \subset H^1(M,\Sigma,\reals)$ denote the $\reals$-linear span of
the vectors $a'$ and $b'$. Since $\cx \tensor H(x')$ coincides
with the tangent space to the $GL(2,\reals)$ orbit of $x'$, we have
$H(x') \subset T_\reals N(x')$. Let $H^\dagger(x') \subset p(T_\reals N(x'))$
denote the symplectic complement to $p(H(x'))$ in $p(T_\reals N(x'))$. 

Since $\nu$ is $SL(2,\reals)$-invariant, the conditional measures of $\nu$
along the $SL(2,\reals)$-orbit are Lebesgue. Also note that for $g \in SL(2,\reals)$ and $x = g x'$, we have
\begin{displaymath}
  TN(x) = TN(x') \qquad H(x) = H(x') \qquad \text{ and } F(x) = F(x'). 
\end{displaymath}
Since the foliation of $N = N(x)$ by $SL(2,\reals)$ orbits is transverse to the
foliation whose leaves are of the the form $x+ TN(x) \cap (\cx \tensor p^{-1}(H^\dagger(x)))$, 
for almost all pairs $x', y$ we can find $x = g x'$ in the support of
$\nu$ such that
\begin{equation}
\label{eq:tmp:y:minus:x}
  y - x \in TN(x) \cap (\cx \tensor p^{-1}(H^\dagger(x)).
\end{equation}
Let $x = a + b \bi$, and let $\delta \in T_\reals N(x)$ 
denote the real part of  $y-x$. 
Then we have by (\ref{eq:tmp:y:minus:x}),
\begin{displaymath}
\langle  p(\delta), p(a)\rangle = \langle p(\delta), p(b) \rangle = 0. 
\end{displaymath}
Also, since $F(x) = F(x')$ and $p(y - x') \not\in \cx \tensor F(x')$ we have 
\begin{displaymath}
p(\delta) \not\in F^\perp(x).   
\end{displaymath}

Recall that by Theorem~\ref{theorem:properties:forni} 
the subspace $F^\perp$ coincides with the
subspace $F^\dagger$ symplectically orthogonal to $F$. Recall
also that $\reals p(a) \oplus \reals p(b) \subset F^\perp$. 

Let $\epsilon > 0$ be sufficiently small. 
Now let $v$ be an arbitrary element of $F(x)$. We will now move $v$
around a square 
\begin{displaymath}
a+b\bi \to (a+\delta)+b\bi \to (a+\delta)+(b+\epsilon a)\bi \to a+
(b+\epsilon a)\bi \to a+b\bi.
\end{displaymath}
In other words we compute
\begin{multline*}
P^-(a+(b+\epsilon a)\bi, a+b\bi)
\, P^+((a+\delta)+(b+\epsilon a)\bi, a+
(b+\epsilon a)\bi) \\
P^-((a+\delta)+b\bi,(a+\delta)+(b+\epsilon a)\bi) \, 
P^+(a+b\bi, (a+\delta) + b\bi) v
\end{multline*}
as indicated in the picture.
\begin{displaymath}
 \begin{xy}
  \xymatrix{
      \overset{a+(b+\epsilon)\bi}{\bullet} \ar[d]_{P^-} & \overset{(a+\delta)+(b+\epsilon)\bi}{\bullet} \ar[l]^{P^+} \\
      \underset{x=a+b\bi}{\bullet} \ar[r]^{P^+}          & \underset{(a+\delta)+b\bi}{\bullet} \ar[u]_{P^-}}
  \end{xy}
\end{displaymath}
\setitemize[1]{itemindent=0pt,leftmargin=12pt,parsep=2pt}
\begin{itemize}
\item Step 1: moving from $a+b\bi$ to $(a+\delta)+b\bi$. 
Note that   since $\delta \in T_\reals N (a+b\bi)$ and $\langle
  p(\delta),p(b)\rangle$ = 0 we conclude $\delta \in W^{uu}(a+b\bi)$
and we can apply Lemma~\ref{lemma:formula:v}.
Using (\ref{eq:formula:v}) with $x = a+b\bi$ and $s = \delta$ we get
\begin{displaymath}
v \to v + \langle v,p(\delta) \, \rangle p(b) =: v_1. 
\end{displaymath}
\item Step 2: moving from $(a+\delta)+b\bi$ to $(a+\delta)+(b+\epsilon
  a)\bi$. We claim that $\epsilon a \in
    W^{ss}(a+\delta + b\bi)$. Indeed, we have $\epsilon a \in T_\reals
    N(a+ b\bi)$ (since $\epsilon a$ belongs to the tangent space to
    the $GL(2,\reals)$ orbit at $a+b\bi$). Since $N = N(x)$ is affine,
    $T_\reals N(a+\delta + b\bi) = T_\reals N(a+b\bi)$, and thus,
    $\epsilon a \in T_\reals N(a+\delta + b \bi)$. Since $\langle
    p(\epsilon a), p(a) \rangle = 0$, we obtain $\epsilon a \in W^{ss}(a +
    \delta + b\bi)$ as claimed.
    Using
  (\ref{eq:formula:v}) with the contracting and expanding directions
  reversed (and taking into account the sign change coming from the
  fact that (\ref{eq:formula:v})   was derived assuming $\langle a, b
  \rangle = +1$)  and $x = (a+\delta)+b\bi$ and $s = \epsilon a$ we get
\begin{align*}
v_1 & \to v + \langle v,p(\delta)
\rangle p(b) - 
\langle v + \langle v,p(\delta) \rangle p(b), \epsilon p(a)\rangle \, p(a+\delta)
\\
& = v + \langle v,p(\delta) \rangle p(b) + \epsilon \langle v,p(\delta)\rangle
\, p(a+\delta) =:v_2. 
\end{align*}
\item Step 3: moving from $(a+\delta)+(b+\epsilon a)\bi$ to
  $a+(b+\epsilon a)\bi$. Again, since $N$ is affine, we have
    $-\delta \in T_\reals N(a+b\bi) = T_\reals N(a+\delta +
    (b+\epsilon a) \bi)$. Since $\langle p(-\delta), p(b+\epsilon a)
    \rangle = 0$, we conclude $-\delta \in W^{uu}(a+\delta+(b+\epsilon
    a)\bi)$.
Using
  (\ref{eq:formula:v}) with $x = (a+\delta)+(b+\epsilon
  a)\bi$ and $s =   -\delta$ we get 
\begin{align*}
v_2  & \to v +  \langle v,p(\delta) \rangle p(b) + \epsilon \langle
v,p(\delta) \rangle
p(a+\delta) \\
 &  \qquad - \langle v + \langle v,p(\delta) \rangle p(b) + \epsilon \langle v,p(\delta)\rangle
p(a+\delta), p(\delta) \rangle p(b+\epsilon a) \\
 &  = v  + \langle v,p(\delta) \rangle p(b) + \epsilon \langle
v,p(\delta)\rangle 
p(a+\delta)  - \langle v,p(\delta)\rangle p(b+\epsilon a) \\
 & = v + \epsilon \langle v, p(\delta)\rangle p(\delta) =: v_3.
\end{align*}
\item Step 4: moving from $a+(b+\epsilon a)\bi$ to
  $a+b\bi$. In this case obviously $\epsilon a \in
    W^{ss}(a+(b+\epsilon a) \bi)$.
Using
  (\ref{eq:formula:v}) with $x = a+(b+\epsilon a)\bi$ and $s = -\epsilon
  a$ (and taking into account the sign change coming from reversing
  $a$ and $b$) we get
\begin{align*}
v_3 & \to  v + \epsilon
 \langle v, p(\delta)\rangle p(\delta) + \langle v + \epsilon
 \langle v, p(\delta)\rangle p(\delta), \epsilon p(a) \rangle p(a) \\
 & = v + \epsilon \langle v, p(\delta)\rangle p(\delta) =:v_4. 
\end{align*}
\end{itemize}
Thus, moving around the square, we map $v$ to $v + \epsilon \langle v,
p(\delta)\rangle p(\delta)$. This contradicts
Lemma~\ref{lemma:parallel:transport:hodge:inv}. 
Therefore (\ref{eq:no:yeti}) holds. (In 
the case when $p(\delta) \not \in F(x)$, then this computation shows that 
moving around the square does not preserve $F(x)$; 
this is a contradiction to the definition of the connection $P^+$.)
\end{proof}

We have proved:
\begin{theorem}
\label{theorem:technical}
There exists a subset $\Psi$ of the stratum $\cH_1(\kappa)$ 
with $\nu(\Psi) = 1$ such
that for all $x \in \Psi$ there exists a neighborhood $U(x)$ such that
for all $y \in U(x) \cap \Psi$ we have $p(y-x) \in F^\perp(x)$. 
\end{theorem}

Theorem~\ref{theorem:technical} is used in \cite{Eskin:Mirzakhani:measures} 
to complete the proof
that any $SL(2,\reals)$ invariant measure is affine.

\begin{proof}[Proof of Theorem~\ref{theorem:forni:transverse:to:tangent}]
Let $\cN$ be an affine submanifold, and let $\nu$ be the affine
measure supported on $\cN$. Then for $\nu$-almost all $x \in \cN$,
$N(x) = \cN$, where $N(x)$ is as in \S\ref{sec:real:analytic:envelope}. 
Then parts (a) and (b) of
Theorem~\ref{theorem:forni:transverse:to:tangent} follow immediately
from Proposition~\ref{prop:no:yeti}. 
Also part (c) of Theorem~\ref{theorem:forni:transverse:to:tangent}
follows immediately from part (a) of Theorem~\ref{theorem:properties:forni}. 
\end{proof}


\begin{proof}[Proof of
Theorem~\ref{theorem:affine:symplectic}] Let $\cN$ be an affine
submanifold, and let $\nu$ be the affine measure supported on
$\cN$. By Proposition~\ref{prop:no:yeti}, then tangent space is
contained in the orthogonal complement of the Forni subspace, i.e.\ 
$T_\reals(\cN) \subset F^\perp$. Then, by 
Theorem~\ref{theorem:properties:forni} (b), the tangent space 
$T_\reals(\cN)$ is
symplectic.   
\end{proof}

\begin{proof}[Proof of
    Theorem~\ref{theorem:locally:constant:semisimple}]
The proof is a lightly modified version of the proof of
  \cite[Theorem A.6]{Eskin:Mirzakhani:measures}.
By Proposition~\ref{prop:no:yeti} and Theorem~\ref{theorem:properties:forni} 
we have the locally constant decomposition
\begin{displaymath}
H^1(M, \reals) = F^\perp \dirsum F,
\end{displaymath}
where both factors are $SL(2,\reals)$-invariant. 
Suppose $L$ is a
locally constant $SL(2,\reals)$-invariant 
subbundle of the Hodge bundle. Let $L^\dagger$ be the symplectic
complement to $L$, and let $L_1  = L \cap L^\dagger$. Then, $L_1$ is
locally constant. Also $L_1$ is isotropic, and therefore
by \cite[Theorem A.4]{Eskin:Mirzakhani:measures}, 
\cite[Theorem A.5]{Eskin:Mirzakhani:measures} and
Theorem~\ref{theorem:forni:subspace:ergodic:def}, 
$L_1 \subset F$.  

Note that by
Lemma~\ref{lemma:parallel:transport:hodge:inv} and
Proposition~\ref{prop:no:yeti}, $F$ is locally constant and 
the monodromy representation
restricted to $F$ has compact image. Therefore there exists a
an $SL(2,\reals)$-invariant locally constant subbundle $L_2$ of $F$
such that $F = L_1 \oplus L_2$. Let $L_3 = L_2 \oplus F^\perp$; then
$L_3$ is also locally constant and $SL(2,\reals)$-invariant. Then,
\begin{displaymath}
L = L_1 \oplus (L \cap L_3), \qquad L^\dagger = L_1 \oplus (L^\dagger
\cap L_3),
\end{displaymath}
and
\begin{displaymath}
H^1(M,\reals) = L_1 \oplus (L \cap L_3) \oplus (L^\dagger \cap L_3).
\end{displaymath}
Thus $L^\dagger \cap L_3$ is an $SL(2,\reals)$-invariant locally
constant complement to $L$.
\end{proof}

\end{document}